\documentclass{amsart}
\usepackage[T1]{fontenc}
\usepackage[english]{babel}
 \usepackage{hyphenat}
\usepackage[utf8]{inputenc}

\usepackage{amssymb,amsmath,amsthm,mathtools}
\usepackage{srcltx}
\usepackage{babel,graphicx}
\usepackage{color}
\usepackage{hyperref}

\newtheorem{theorem}{Theorem}[section]
\newtheorem{proposition}[theorem]{Proposition}

\newtheorem{lemma}[theorem]{Lemma}

\theoremstyle{definition}
\newtheorem{definition}[theorem]{Definition}

\newcommand{\Q}{\mathbb{Q}}

\newcommand{\HH}{\mathbb{H}}
\newcommand{\e}{\epsilon} 

\newcommand{\fix}{\text{Fix}}
\newcommand{\vol}{\text{Vol}}

\title[Finiteness of arithmetic reflection groups]{A new proof of finiteness of maximal arithmetic reflection groups}

\author{David Fisher and Sebastian Hurtado}


\thanks{D.F. was supported by NSF Grant DMS-1906107 and the Miller Institute at Berkeley. S.H. was supported by the Sloan Fellowship Foundation.}

\begin{document}
\maketitle


\section{Introduction}

In this note, we offer a new proof of the following result:

\begin{theorem} \label{main0} For every $n \geq 2$, there exists only finitely many maximal arithmetic reflection groups $\Gamma \subset \text{Isom}(\HH^n)$.
\end{theorem}

\noindent Nikulin proved in \cite{Nikulin1} that there were at most finitely many maximal arithmetic reflection groups in each dimension $n$ where the degree of the adjoint trace field is fixed.  Therefore Theorem \ref{main0} follow from bounding the degree of the number field and our main contribution is to give a new proof of the following:

\begin{theorem}\label{main1}
For every $n \geq 2$, there exists $d_n > 0$, such that if $\Gamma$ is arithmetic reflection group in $\HH^n$, then degree of the adjoint trace field of $\Gamma$ over $\Q$ is at most $d_n$.
\end{theorem}

\noindent The fact that Theorem \ref{main0} implies finiteness of all maximal arithmetic reflection groups
can be deduced in a variety of ways from prior work of Nikulin, Prokhorov and Vinberg, see the introduction
to \cite{ABSW} or the survey \cite{BelolipetskySurvey}.  We note here that we do not use that the reflection
group is maximal in our proof of Theorem \ref{main1} nor do we need the fact that maximal arithmetic reflection
groups are congruence.

We briefly discuss the prior history of finitenes of maximal arithmetic reflection groups, referring the reader to Belolipetsky's excellent survey for more details on this and other related points \cite{BelolipetskySurvey}.  In \cite{LMR}, Long, Maclachlan and Reid prove finiteness of arithmetic surfaces of genus zero, which implies the desired result.  Shortly afterwards, Agol proved finiteness of maximal arithmetic Kleinian reflection groups in \cite{Agol}.  Following that work there were two independent proofs of finiteness in higher dimensions, one by Nikulin and one by Agol, Belolipetsky, Storm and Whyte \cite{Nikulin2, ABSW}.  Roughly \cite{ABSW} generalize the proofs of \cite{Agol, LMR} to higher dimensions using some new inputs, while \cite{Nikulin2} uses older reflection group technology from Nikulin's prior work on finiteness to prove the case of general dimension by induction on the results in dimension $2$ and $3$.  In either case, the proofs relies in the end in a central way on deep results in the theory of automorphic forms that provide an absolute lower bound on the first eigenvalue of the Laplacian on the relevant orbifolds. Our main motivation in writing this note is to give a proof that is independent of the spectral bounds and the theory of automorphic forms.

Our main tool will be the following:

\begin{lemma}[Arithmetic Margulis Lemma]\label{AML}
For every $n > 0$, there exists $\epsilon_n$ such that if $\Gamma \subset \text{Isom}(\HH^n)$ is an arithmetic group whose trace field has degree $d$ over $\Q$,
then for every $x \in \HH^n$, the group: $$\Gamma_x := \langle \{ \gamma \in \Gamma | d_{\HH^n}(\gamma x, x ) \leq \epsilon_n d \} \rangle$$ is virtually solvable.
\end{lemma}

\noindent This lemma is proven in \cite{FHR} using Breuillard's height gap theorem from \cite{Breuillard}.  Breuillard's
proof of the height gap theorem is also far from elementary, but a recent elementary proof was obtained by the second author with Chen and Lee \cite{CHL}.  In particular, our proof of Theorem \ref{main0} does not rely on heavy machinery from number theory or representation theory in any way.  We thank Ian Agol for suggesting that it might be interesting to find such a proof.

We mention briefly here that once Theorem \ref{main1} was known, further work has been done on explicitly computing $d_n$. For a fairly up to date account of these developments see \cite[Section 5]{BelolipetskySurvey}.  Only $d_2$ has been computed exactly in more recent work of Linowitz \cite{Linowitz}.  To achieve anything in this direction by our methods would require an effective version of Lemma \ref{AML} and so an effective version of the height gap theorem.  This seems quite difficult in general, but could perhaps be easier in the explicit settings required for the study of hyperbolic reflection groups.

We finish this introduction by giving an outline of the Proof of Theorem \ref{main1} in the case $n=2$. Assume $\Gamma$ is an arithmetic reflection group generated by reflections on the sides of an acute-angled polygon $P$ in $\HH^2$.  Let $\mathcal{E} = \{e_1, e_2, \dots , e_k\}$ be the edges of $P$ and consider the collection of balls $ \mathcal{B} = \{B_1, B_2, \dots , B_k\}$ of radius $\frac{\e_2 d}{2}$ centered at the midpoints of the edges of $P$,  where $\e_2$ is the constant in \ref{AML} and $d$ is the degree of the trace field of $\Gamma$.

If a ball $B_i \in \mathcal{B}$ intersects three edges (or intesect two other balls in $\mathcal{B}$) , the group generated by reflections on these three sides will be typically a non-virtually solvable group\footnote{unless the three edges are adjacent and meet at right angles, a problematic case that requires further considerations}, therefore if $d$ is sufficiently large, Lemma \ref{AML} implies the balls in $B$ typically cannot have triple intersections and intersect at most two edges of $P$. To illustrate the argument, assume there are no triple intersections of balls of $\mathcal{B}$ and that each ball in $\mathcal{B}$ intersects only one edge, and so half of each ball is contained in $P$. This implies that the area of $P$ has to be greater than $\frac{k}{4} \vol(B_{\HH^2}(\e_2d))$, but elementary hyperbolic geometry shows that $P$ has volume at most $(k-2)\pi$. If $d$ is sufficiently large both inequalities are not possible. In the case $n \geq 3$, we will apply a similar argument to a two-dimensional face of $P$, in this case we will also need to make use of the simplicity of such polyhedra due to Vinberg and a Theorem of Andreev.

\noindent
{\bf Remark:} Very shortly after we shared a draft of this paper, Jean Raimbault replied
with an even shorter proof of Theorem \ref{main0}.  Raimbault's proof uses \cite[Theorem D]{FHR}
and a remarkable new result: for any reflection hyperbolic manifold in dimension n, the thin part
has at least a fixed proportion of the volume \cite{Raimbault}.  Since \cite[Theorem D]{FHR} uses the trace formula,
we remark here that it is also possible to produce a proof of Theorem \ref{main0} using \cite[Theorem C]{FHR}, Raimbault's
lower bound on the volume of the thin part, and Nikulin's results from \cite{Nikulin1}.  While Raimbault's
proof is shorter and more conceptual than ours,  it seems interesting to publish our approach
as the proof is more elementary and the ideas might prove useful for a variety of problems concerning real and complex
hyperbolic lattices generated by torsion elements.

\section{Proof of Theorem \ref{main1}}

\begin{definition} For a hyperplane $H$ in $\HH^n$, let $I_H$ be the involution in $\HH^n$ with respect to $H$. Let define $\fix(H)$ to be the set of fixed points of $I_H$ in the boundary $\partial \HH^n$.
\end{definition}

Let $P$ be the acute-angled hyperbolic polytope defining $\Gamma$, and so there exists a collection of oriented hyperplanes $\{H_{\alpha}\}_{\alpha \in I}$, such that $P = \bigcap H^{+}_{\alpha}$ and $\Gamma = \langle \{I_{H_{\alpha}} \}_{\alpha \in I}\rangle$.  We will argue by contradiction and suppose that $\Gamma$ is defined over a number field of degree $d \geq d_n$, where $d_n$ is the smallest positive integer such that the ball $B$ of radius $d_n\epsilon_n/4$ in $\HH^2$ satisfies :

\begin{equation}\label{dn}
  \vol_{\HH^2}(B_{d_n\epsilon_n/4}) \geq 12\pi
 \end{equation}

\noindent where  $\epsilon_n$ is the constant in the Arithmetic Margulis Lemma \ref{AML}.

We will make use of the following results of Andreev \cite{Andreev} and Vinberg about the polyhedron $P$:

\begin{theorem}[Vinberg \cite{Vinberg}, Cor. Thm 3.1]\label{simplicity} Let $P$ be an acute-angled polytope in $\HH^n$. Then $P$ is \emph{simple}, meaning that  for every $k$, each face of codimension $k$ is contained in exactly $k$ codimension-one faces.
\end{theorem}

\begin{theorem}[Andreev]\label{Andreev}
Let $P$ be an acute-angled polyhedron and suppose $\{F_i\}$ is a (finite or countable) collection of codimension one faces of $P$
and $\{H_i\}$ is the corresponding collection of hyperplanes in $\HH^n$.  Then

$$\dim(\cap _{i\in I} \overline{F_i}) = \dim({\cap_{i \in I} \overline{H_i}}).$$
\end{theorem}

\noindent Here the closures occurring in the statement occur in $\overline{\HH^n}$ and we assume a point in the boundary of $\HH^n$ has dimension $-1$ and that the empty set has dimension $-\infty$.  One can read Andreev's theorem as saying the faces
of the polygon do not have ``extra" intersections not already occurring in $\overline{P}$.

We will use repetitively the following elementary facts:

\begin{proposition}\label{basic1}
If $\Gamma $ is a virtually solvable group of $\text{Isom}(\HH^n)$, then there exists $x \in \HH^n \cup \partial \HH^n$ such that $|\Gamma x| \leq 2$. Moreover if $\Gamma$ is not finite, then $x \in \partial \HH^n$.
\end{proposition}

\begin{proposition}\label{basic2}
Suppose $H_1$ and $H_2$ are two hyperplanes in $\HH^n$, and suppose that  $H_1 \cap H_2 \neq \emptyset$, then if $ x \not\in \fix(H_1) \cup \fix(H_2) $ and $\Gamma:=  \langle I_{H_1}, I_{H_2} \rangle$ is the subgroup generated by reflections in $H_1, H_2$, then $| \Gamma x | \geq 3$.
\end{proposition}

We begin the proof of Theorem \ref{main1}. Choose a 2-dimensional face $F$ of $P$, and cyclically enumerate its collection of edges $e_0, e_1, \dots, e_m$ in such a way that $e_{i}$ is adjacent to $e_{i+1}$. By the simplicity of $P$ we have that $F$ is contained in exactly  $H_1, H_2, \dots, H_{n-2}$ hyperplanes determining faces of $P$ and so $F = (\cap_{j=1}^{n-2} H_{j}) \bigcap P $, and for each edge $e_i$ there exists a unique hyperplane $H_{e_i}$ such that $e_i = F \cap H_{e_i}$.

Our main technical tool in the proof is the following:

\begin{lemma}\label{maintool} Let $e_i, e_j, e_k$ be three different edges of $F$, let $$\Gamma_0 :=  \langle I_{H_{e_i}}, I_{H_{e_j}}, I_{H_{e_k}}, I_{H_1}, \dots, I_{H_{n-2}} \rangle$$ then at least one of the following holds:

\begin{enumerate}
\item $\Gamma_0$ is not virtually solvable.
\item The edges $e_i, e_j, e_k$ are all adjacent, and  so up to reordering $j= i+1, k = i+2$. Moreover the angles at $e_i\cap e_{i+1}$, and $e_{i+1}\cap e_{i+2}$ are both $\frac{\pi}{2}$.

\end{enumerate}
\end{lemma}

\begin{proof}
Observe that $\Gamma_0$ is not finite, otherwise $\Gamma_0$ has a fixed point in the interior that must be in the intersection of $H_1, \dots H_{n-2}, H_{e_i}, H_{e_j}, H_{e_k}$, but that would contradict the simplicity of  $P$ (Lemma \ref{simplicity}).

We will show that if item 2) does not hold then $\Gamma_0$ is not virtually solvable. By Lemma \ref{basic1} we have to show that if  $x \in \partial \HH^n$, we must have $|\Gamma_0 x| \geq 3$

Let $$\HH^2 := H_1 \cap H_2 \dots \cap H_{n-2}$$ and consider the following cases:

\textbf{Case 1: } Suppose $x \not \in \HH^2$, and so up to reordering suppose that $I_{H_{n-2}}(x) \neq x$. If either $ I_{H_{e_i}}(x),  I_{H_{e_j}}(x),  I_{H_{e_k}}(x)$ are different than $x$, as $H_{n-2}$ intersect the hyperplanes $H_{e_i}, H_{e_j}, H_{e_k}$, then we can apply \ref{basic2} to show $|\Gamma_0x| \geq 3.$ Therefore $ I_{H_{e_i}}(x) =  I_{H_{e_j}}(x) =  I_{H_{e_k}}(x) = x$, moreover as $ I_{H_{n-2}}(x)$ is also not fixed by $H_{n-2}$, the same argument shows that $ I_{H_{n-2}}(x)$ is fixed by $ I_{H_{1}}, \dots  I_{H_{n-2}},  I_{H_{e_i}},  I_{H_{e_j}},  I_{H_{e_k}}$, but then the line $l := \overline{x I_{H_{n-2}}(x)} $ must be fixed by $ I_{H_{1}}, \dots  I_{H_{n-3}},  I_{H_{e_i}},  I_{H_{e_j}},  I_{H_{e_k}}$, and by Andreev's Theorem we must have that $l$ contains a one dimensional face $e$ of $P$, contradicting the simplicity of $P$ at $e$.

\textbf{Case 2: } Suppose $x \in \HH^2$. If $ I_{H_{e_i}}(x) \not\in \HH^2$, then $I_{H_{e_i}}(x)$ is not fixed by some $ I_{H_{l}}$, and as $H_{e_i}$ intersects $H_{l}$, Lemma \ref{basic2} implies that $|\Gamma_0(x)| \geq 3$. Similarly we can assume that   $ I_{H_{e_j}}(x),  I_{H_{e_k}}(x) \in \HH^2$. Let $l_i, l_j, l_k$ be the lines containing $e_i, e_j, e_k$ respectively.

Suppose $x \not\in \fix(H_{e_i}) \cup \fix(H_{e_j}) \cup \fix(H_{k})$ and $|\Gamma_0 x| = 2$, then the line $l := \overline{x I_{H_{e_i}}(x)}$ must be orthogonal to all three of  $H_{e_i}$, $H_{e_j}$, $H_{e_k}$ and each of them contains a codimension one face of $P$, and this contradicts the convexity of $P$. Therefore we can assume that $ I_{H_{e_i}}(x) = x$.  If $l_j$ (similarly $l_k$) intersects $l_i$ and the angle at the intersection is not $\frac{\pi}{2}$, it follows that all  $x,  I_{H_{e_j}}(x)$ and $ I_{H_{e_i}} I_{H_{e_j}}(x)$ are all distinct and we are done. Therefore we have (up to renaming) that either:

\textbf{Case 2a:} The line $l_i$  intersects $l_k$ at  angle $\frac{\pi}{2}$ but $l_i$ does not intersect $l_j$ (in the interior of $\HH^2$).  In this case the set $ \{x,  I_{H_{e_k}}(x), I_{H_{e_j}}(x), I_{H_{e_j}}(I_{H_{e_k}}(x)) \}$ contains at least three diferent points and so  $|\Gamma_0(x)| \geq 3$.

\textbf{Case 2b:} The line $l_i$ does not intersect  $l_j \cup l_k$ in the interior of $\HH^2$. By Andreev's theorem \ref{Andreev} again, if $l_i$ intersects either $l_j$ or $l_k$ at infinity, it does so in a vertex at infinitey of $P$ and then by convexity only one of the two intersections can occur at $x$. In this case, we can suppose that $ I_{H_{e_k}}(x) \neq x$, and we have that that  $ I_{H_{e_k}}(x)$ is not fixed by  $I_{H_{e_i}}$ therefore $x, I_{H_{e_k}}(x), I_{H_{e_k}}I_{H_{e_i}}(x) $ are all distinct. We remark that this case can be eliminated by other methods, since non-compact arithmetic reflection groups all have trace field $\mathbb{Q}$.

\textbf{Case 2c:} The line $l_i$ intersects both $l_j$ and $l_k$ perpendicular. In this case Andreev's Theorem \ref{Andreev} implies that $e_i, e_j, e_k$ are adjacent and perpendicular as in item 2) and we are done.

\end{proof}

\begin{definition}
We say that an edge $e_i$ of $F$ is small (or large) if the hyperbolic length $l(e_i) \leq \frac{\e d}{2}$ ( $l(e_i) > \frac{\e d}{2}$), where $\e$ is the constant in the Arithmetic Margulis Lemma \ref{AML}.
\end{definition}

\begin{proposition}\label{ssnotallowed} There are not two consecutive small edges $e_i, e_{i+1}$ in $F$.
\end{proposition}

\begin{proof}

Let's argue by contradiction, suppose that $i=2$ and that $e_i, e_{i+1}$ are small, and  assume that  $F$ has more than five edges. Let $$\Gamma_0 := \langle  I_{H_{e_1}}, I_{H_{e_2}}, I_{H_{e_3}}, I_{H_{e_4}},I_{H_1}, \dots I_{H_{n-2}} \rangle. $$
\noindent By the arithmetic Margulis lemma $\Gamma_0$ applied to the vertex $e_2\cap e_3$, $\Gamma_0$ is virtually solvable. We can also apply Lemma \ref{maintool} to $e_1, e_2, e_4$ and conclude $\Gamma_0$ is not virtually solvable, obtaining a contradiction.  For the case of $F$ having $3$ or $4$ sides, the argument is similar and simpler.  One considers the vertex where the two consecutive short edges meet and notices that the arithmetic Margulis lemma implies the group generated by the reflections giving all faces of $F$ is virtually solvable, which is easy to contradict.

\end{proof}

So from now on we can assume that there are no two small consecutive edges in $F$, and so at least $m/2$ among $e_1, e_2,  \dots, e_m$ are large.


\begin{proposition}\label{balls} If $e_i$ is large, there exists $q_i \in e_i$ such that the ball $B_i$ with center $q_i$ and radius $\e d/4$ satisfies that  half of $B_i$ is contained in $F$,  more precisely $\text{Vol} (B_{i} \cap F) = \frac{1}{2} \text{Vol}(B_i)$.
\end{proposition}

\begin{proof}

Consider a closed ball $B_i$ of radius $\e d /4$ centered at the midpoint $p_i$ of $s_i$. We will slide the ball $B_i$ along $s_i$ (moving the center $p_i$ in the edge $e_i$) until half of $B_i$ is totally contained in $F$ as follows:

Let $x_i, y_i$ be the vertices of $e_i$. If $B_i$ does not intersect another edge of $F$ we are done. Suppose that $B_i$ intersects another edge $e_j  \neq e_i$, we start moving $B_i$ as follows:

\textbf{Case 1}: If $e_j$ is adjacent to $e_i$, say $e_j = e_{i-1}$, we slide $B_i$ towards $y_i$, until either the interior of half of $B_i$ is totally contained in $F$ and we are done, or until at some point $B_i$ intersects three different edges including $e_{i-1}$(tangencies count), so we have that $B_i$ intersects $e_i, e_{i-1}$ and $e_k$, $k \neq i, i-1$, and moreover the angle at $x_i$ is less than 90 degrees. By the Arithmetic Margulis Lemma applied to the center of $B_i$, we have  that $\Gamma_0:= \langle  I_{H_{e_i}},  I_{H_{e_{i-1}}},  I_{H_{e_k}},  I_{H_{1}}, \dots  I_{H_{n-2}}\rangle$ is virtually solvable, contradicting Lemma \ref{maintool}, and so this triple intersection is not allowed.

\textbf{Case 2}: The idea is the same. If $e_j$ is not adjacent to $e_i$, we slide $B_i$ towards $x_i$ until half of $B_i$ is totally contained in $F$, or until half of $B_i$ intersects another edge $e_k$. By the Arithmetic Margulis Lemma applied to the center of $B_i$, we have  that $\Gamma_0:= \langle  I_{H_{e_i}},  I_{H_{e_j}},  I_{H_{e_k}},  I_{H_{1}}, \dots  I_{H_{n-2}}\rangle$ is virtually solvable, and so by Lemma \ref{maintool} $e_j, e_k, e_i$ are consecutive and the angles at $e_j\cap e_k, e_k \cap e_i $ are $\frac{\pi}{2}$, we now slide $B_i$ towards $y_i$ until half of $B_i$ is either totally contained in $P$ or we have that $B_i$ intersects $e_i, e_j$ and another edge $e_l$ (different that $e_k$). In this case, again by the Arithmetic Margulis Lemma and Lemma \ref{maintool}, we have $e_i, e_l, e_j$ are adjacent and perpendicular.  This then implies that $F$ has four edges and all angles $\frac{\pi}{2}$, which is impossible.

\end{proof}

We can now finish the proof of Theorem \ref{main1} by the following volume considerations.

For every large edge $e_i$ consider the half ball $B_i\cap F$, observe that by the Arithmetic Margulis Lemma and Lemma \ref{maintool}, four of these half balls can't intersect non-trivially and so recalling that $F$ have $m$ edges and at least $m/2$ large edges, we have:  $$ 3\vol(F) \geq \sum_{i  \ \text{large}} \vol (B_i\cap F)  \geq \frac{k}{2} \frac{\vol(B_1)}{2} $$

And by elementary hyperbolic geometry $(k-2)\pi \geq \vol(F)$, which implies that $12\pi > vol(B_1)$ and contradicts that $d \geq d_n$ by  equation \ref{dn}.

\end{document}